\documentclass[12pt]{article}
\usepackage{cite}
\usepackage{amsmath,amsthm,amssymb}
\usepackage[left=1.3in,
right=1.3in,nohead]{geometry}
\usepackage{hyperref}
\usepackage{latexsym}
\usepackage[sc,osf]{mathpazo}

\usepackage{pgf,tikz}
\usetikzlibrary{arrows,cd,matrix,decorations.pathmorphing}
\usetikzlibrary{babel}            

\def\bct{\begin{center}}
\def\ect{\end{center}}
\def\beg{\begin}

\def\<{\langle}
\def\>{\rangle}

\def\mbb{\mathbb}
\def\mbbcp{\mathbb{CP}}
\def\mbbr{\mathbb R}
\def\mbbz{\mathbb Z}

\def\ni{\noindent}
\def\ol{\overline}

\def\tcb{\textcolor{blue}}
\def\tn{\textnormal}

\newtheorem{thm}{Theorem}[section]

\title{A family of cohomological 
complex projective spaces}
\author{Mustafa Kalafat}

\begin{document}
\maketitle


\begin{abstract}
We exhibit 
a family of complex manifolds, which has a member at each odd complex dimension
and which has the same cohomology groups 
as 
the complex projective space at that dimension, 
but not homotopy equivalent to it. 
We also analyze the even dimensional analogue.

\vspace{.05in}
\ni {\em Keywords:} Grassmannians; Algebraic Topology; Serre spectral sequence. 

\vspace{.05in}

\ni {\em Mathematics Subject Classification 2010:} Primary 53C25; Secondary  53C55.

\end{abstract}

\section{Introduction}\label{secintro}

This paper is about various examples of 
complex manifolds. 
Conventionally we define the complex degree $2$ hypersurface in the n-dimensional complex projective space by 
$$\mbb V_2:=\{[Z_0:\cdots :Z_n]\,|\, Z_0^2+\cdots+Z_n^2=0\}\subset \mbbcp_n.$$ 

 Our main result is as follows. 
\begin{thm} The algebraic variety $\mbb V_2\subset\mbbcp_{2k}$ has the
same cohomology groups 
as $\mbbcp_{2k-1}$ but it is not homotopy equivalent to it for $k>1$. 
\end{thm}

\ni The idea of the proof is to work on the underlying 
smooth structure of this hypersurface. One can show that \cite{redbook} this complex manifold is the set of totally isotropic complex 2-dimensional subspaces of 
$\mbb C^{n+1}$ and 
indeed diffeomorphic to the oriented Grassmannian $G_2^+\mbbr^{n+1}$, the space of oriented 2-planes in real n+1-space. In the lower dimension for k=1, 
$\mbb V_2$ is the projective line and $G_2^+\mbbr^3$ is a topological sphere, so 
all these spaces are the same. However if one goes further in the upper dimensions the spaces start diverging from each other. 
To be able o distinguish, furthermore using Serre spectral sequence techniques one can 
compute the 
cohomology ring of the Grassmannian as follows.

\noindent {\bf Theorem \ref{thmoddcuptruncated}.} 
{\em 
The cohomology ring of the Grassmann manifold is the following truncated polynomial ring for which $deg x_m=m$. 
\[ H^*( G_2^+\mbbr^{2k+1};\mbbz) \approx  
\mbbz[x_2,x_{2k}]/ \langle x_2^k-2x_{2k}, x_{2k}^2\rangle. \]}   
\ni One can compare the cohomology rings to distinguish this space with 
the projective space or one can use homotopy groups. We give the details 
in the following. To give a complete discussion we also give the cohomology ring in the even case as follows.

\noindent {\bf Theorem \ref{thmevencuptruncated}.} 
{\em The cohomology ring of the Grassmann manifold is the following truncated polynomial ring for which $deg x_m=m$. 
\[\hspace{-1cm}
H^*( G_2^+\mbbr^{2k};\mbbz) \approx  
{\mbbz[x_2,x_{2k-2},y_{2k-2},x_{2k}] \tn{\Large $/$} \hspace{-1.5mm} \left\langle 
\begin{array}{c} 
x_{2k-2}+y_{2k-2}+(-1)^k x_2^{k-1}~,~ 
2x_{2k-2}^2-x_2^{2k-2} \\ [1.5\jot] 
2x_2x_{2k-2}+(-1)^k x_2^k~,~2y_{2k-2}^2 - x_2^{2k-1}\\ [1.5\jot]
x_{2k-2}y_{2k-2}~,~x_{2k}^2 ~,~ x_2^{2k-1}
\end{array}  \right\rangle  }. \]
}

There has been a growing interest in Grassmann manifolds because of their 
role in analysing submanifolds of smooth manifolds or Riemannian manifolds 
with special or exceptional holonomy. This is another reason that makes 
our work interesting.  Interested reader may consult to \cite{atg2}, \cite{coassf} and \cite{slag} for more information. 

This paper is organized as follows. In section \S\ref{secmain} we deal with the case n=7, in section \S\ref{secodd} we deal with the odd and in section \S\ref{seceven} with the even dimensional analogues. 

\noindent\textbf{Acknowledgements.} 
The  author would like to thank to his father and family for their support during the preparation of this paper. 
This work is partially supported by the grant $\sharp$114F320 of T\"ubitak \footnote{Turkish science and research council.}.

\newpage
\section{Typical Case}\label{secmain}					

In order to work on the Grassmannian we need the setting of a cohomological Serre spectral sequence of a fiber bundle. The reader may consult to \cite{hatcherss} or \cite{hajimesato} for the fundamental tools in spectral sequences.  
We will be working on the following fiber bundle  
\begin{equation} 
\mbb S^1 \rightarrow V_2\mbbr^7 \longrightarrow G_2^+\mbbr^7 \label{G2R7fibration}
\end{equation}
\ni obtained by sending an ordered orthonormal $2$-frame of $\mbbr^7$ to the oriented $2$-plane spanned by it. One could 
work with any ordered $2$-frames as well, which constitutes a space that deformation retracts to our $V_2\mbbr^7$. Having a simply connected cell complex as a base is crucial as one can show in (\ref{homotopyG2R7}).  
Since we want to understand the cup product structures we are interested in cohomology. In this case the Serre spectral sequence  
is defined as 
\begin{equation*}E^{p,q}_2:= H^p( G ; H^q(\mbb S^1;\mbbz) )\end{equation*}	 
\ni which is guaranteed to converge to the following limit
\beg{equation*}E^{p,q}_\infty \approx F^{p,q}/F^{p+1,q-1} 
\end{equation*}
\ni where $F^{p,q}$ is a  filtration of abelian groups satisfying, 
\begin{equation*} H^n(V;\mbbz) =F^{0,n}\supset F^{1,n-1}\supset \cdots \supset
F^{n+1,-1}=0. \end{equation*}


	
\ni We have a 
product $E_{n}^{p,q}\otimes E_{n}^{p^{\prime },q^{\prime }}\rightarrow
	E_{n}^{p+p^{\prime },q+q^{\prime }}$ on the pages of the spectral sequence satisfying a Leibniz rule, 	
\beg{equation}d_n(ab) =(d_n a)b + (-1)^{p+q}ad_n b~~\tn{for}~~ a\in E_n^{p,q},  b\in E_n^{p^\prime,q^\prime}.\label{leibnizofsss}\end{equation}
\ni On the second page this product, 
$$\hspace{-13mm}E_2^{p,q}( \approx H^p( B;H^q( F;\mbbz) )) \otimes E_{2}^{p^{\prime },q^\prime}( \approx H^{p^\prime}( B;H^{q^\prime}( F;\mbbz) ) )
\rightarrow E_2^{p+p^\prime,q+q^\prime } ( \approx H^{p+p^\prime}( B;H^{q+q^\prime}( F;\mbbz)) ) $$			
\ni equals $\left( -1\right) ^{qp^{\prime }}$ times the cup product of the base space $B$, whose coefficient product is the cup product of the fiber $F.$ 
The product for $E_{n+1}$ is derived from the product for $E_{n}$, 
and the product for $E_{\infty }$ is derived from the cup product for $E.$

The homology groups of various Grassmann manifolds are 
computed in \cite{atg2} using the Serre spectral sequence. 
We can reinterpret one of the results in there through Poincar\'e duality 
and write the table for the cohomology of the Grassmann manifold $G_2^+  
\mbbr^7$ of oriented 2-planes in real 7-space as follows, 		
\begin{align}
H^\ast(G_2^+\mbbr^7;\mbbz)=(\mbbz,0,\mbbz,0,\mbbz,0,\mbbz,0,\mbbz,0,\mbbz). \label{cohomology of G2R7} \end{align}

\ni Next, we will set up the cohomological Serre spectral sequence of the fiber 
bundle (\ref{G2R7fibration}). Imposing the definition and using 
(\ref{cohomology of G2R7}) we obtain the second page as illustrated in the Table 
\ref{table:E2CohomG27}.


\beg{table}[ht]
\caption{ {\em The second page of the cohomological Serre spectral sequence for $G_2^+\mbb R^7$.}}
\bct{
$\begin{array}{cc|c|c|c|c|c|c|c|c|c|c|c|} \cline{3-13}
    &1&\mbbz a&0&\mbbz ax_2&0&\mbbz ax_4&0&\mbbz ax_6&0&\mbbz ax_8&0&\mbbz ax_{10} \\ \cline{3-13}
E_2~~~ &0&\mbbz 1&0&\mbbz x_2&0&\mbbz  x_4&0&\mbbz x_6&0&\mbbz x_8&0&\mbbz x_{10}\\ \cline{3-13}
& \multicolumn{1}{c}{ } & \multicolumn{1}{c}{0} 
& \multicolumn{1}{c}{1} & \multicolumn{1}{c}{2} 
& \multicolumn{1}{c}{3} & \multicolumn{1}{c}{4}
& \multicolumn{1}{c}{5} & \multicolumn{1}{c}{6} 
& \multicolumn{1}{c}{7} & \multicolumn{1}{c}{8}
& \multicolumn{1}{c}{9} & \multicolumn{1}{@{}c@{}}{10} \\  
\end{array}$   
\setlength{\unitlength}{1mm}
\begin{picture}(0,0)
\put(-98,5.5){\vector(2,-1){8}}
\end{picture} }
\ect\label{table:E2CohomG27}\end{table}
\ni	To be able to figure out the cup product structure, we have to find the limit of this spectral sequence. For this reason, we need information on the total space of the fiber bundle. The cohomology of the 11-dimensional Stiefel manifold $V_2\mathbb{R}^{7}$ is well known (see \cite{hatcherat} for a reference) and is given by, 
\begin{equation*} H^{\ast}(V_2 \mbbr^7;\mbbz)
\approx\mbbz[x_{11}]/\langle x_{11}^2\rangle\oplus\mbbz_2[x_6]/\langle x_6^2\rangle ~~~\tn{where}~~~ {deg}\,x_k 
=k. \label{V2R7cohomologyring}\end{equation*}
		
\ni Since the total space at the levels 1,2,3,4,5 and 
7,8,9,10 has vanishing cohomology, the limiting diagonals at these levels are totally zero. Let us figure out the remaining diagonals. 
The term $E_\infty^{5,1}=0$ automatically since it is the limit of $E_2^{5,1}=H^5(G;\mbbz)=0$. Since, 
$$\mbbz_2=H^6(V;\mbbz)=F^{0,6}\supset F^{1,5}\supset\cdots\supset F^{5,1}\supset
F^{6,0}\supset 0$$
\ni and $E_\infty^{0,6}=\cdots=E_\infty^{4,2}=0$ from the second page 
yields $\mbbz_2=F^{1,5}=\cdots=F^{6,0}$. Then we can compute the limit 
$E_\infty^{6,0}=F^{6,0}/F^{7,-1}=\mbbz_2.$

\ni Next, in a similar fashion starting with 
$\mbbz=H^{11}(V;\mbbz)=F^{0,11}$ 
and using $E_\infty^{0,11}=\cdots=E_\infty^{9,2}=0$ from the second page 
we get $\mbbz=F^{1,10}=\cdots=F^{10,1}$. 
Again, from the second page we have $E_\infty^{11,0}=0$. 
Finally, we can compute the limit 
$E_\infty^{10,1}=F^{10,1}/F^{11,0}=\mbbz.$ 
\ni Combining these results we obtain Table \ref{table:EinftyCohomG27} as the 
limit of the spectral sequence. 
\beg{table}[ht]
\caption{{\em The limiting page of cohomological Serre spectral sequence for $G_2^+\mbb R^7$.}}
\bct{
$\begin{array}{cc|c|c|c|c|@{}c@{}|@{}c@{}|c|c|c|c|c|c} \cline{3-13}
           &1&0&0&0&0&0
           &\tcb 0&0&0&0&0&\tcb\mbbz \\ \cline{3-13}
E_\infty~~~&0&\mbbz&0&0&0&   0   &0
&\tcb{\mbbz_2}&0&0&0&0 & \tcb 0 \\ \cline{3-13}
& \multicolumn{1}{c}{ } & \multicolumn{1}{c}{0} 
& \multicolumn{1}{c}{1} & \multicolumn{1}{c}{2} 
& \multicolumn{1}{c}{3} & \multicolumn{1}{c}{4}
& \multicolumn{1}{c}{5} & \multicolumn{1}{c}{6} 
& \multicolumn{1}{c}{7} & \multicolumn{1}{c}{8}
& \multicolumn{1}{c}{9} & \multicolumn{1}{@{}c@{}}{10} \\  
\end{array}$   }\ect \label{table:EinftyCohomG27}  \end{table}
\ni Now we are ready to compute the ring structure of the cohomology. 
Let the symbols $a$ and $x_k$ denote generators of the groups $E_2^{0,1}$ and 
$E_2^{k,0}$ all of which are isomorphic copies of $\mbbz$ as shown in Table \ref{table:E2CohomG27}. The generators on the upper row are $a$ times the generators of the lower row since the product $E_2^{0,q}\times E_2^{r,s}\to E_2^{r,s+q}$ is nothing but multiplication of coefficients. 
Realize that the differentials $d_2^{0,1},d_2^{2,1}$ and $d_2^{6,1},d_2^{8,1}$ 
are isomorphisms. 
So if we rearrange the sign of $a$ or $x_2$ to make the relation $d_2a=x_2$ hold, 	
we can use the Leibniz rule and rearrange the sign of $x_4$ to make the following relation hold, 
\begin{align}x_4=d_2(ax_2)=(d_2a)x_2\pm a(d_2x_2)=(d_2a)x_2=x_2^2\label{ringrel2}\end{align} 	
\ni since $d_2x_2$ is out of the table. Since the limit $E_\infty^{6,0}=\mbbz_2$, 
the differential $d_2^{4,1}$ has to double the generators. So that again applying the Leibniz rule and rearranging we impose,		
\begin{align}2x_6=d_2(ax_4)=(d_2a)x_4=x_2x_4=x_2^3.\label{ringrel3}\end{align}
The remaining differentials are also isomorphisms so we get, 
\begin{align} x_8  =d_2(ax_6)=(d_2a)x_6=x_2x_6\label{ringrel4} \end{align}
\begin{align}x_{10}=d_2(ax_8)=(d_2a)x_8=x_2x_8=x_2^2x_6\label{ringrel5}\end{align}

\ni So, combining the relations (\ref{ringrel2},\ref{ringrel3},\ref{ringrel4},\ref{ringrel5}) we can conclude that 
the cohomology ring of the Grassmann manifold is the following truncated polynomial ring for which $deg x_k=k$. 
\[ H^*( G_2^+\mbbr^7;\mbbz) \approx  
\mbbz[x_2,x_6]/\langle x_2^3-2x_6, x_6^2\rangle. \]	
	
\ni One can also compute the ring structure of $\mbbcp_5$ in a similar fashion 
using Serre spectral sequence as in \cite{botttu} to get  
$$H^*(\mbbcp_5;\mbbz)=\mbbz[x_2]/\langle x_2^6 \rangle$$
\ni where 
$x_2$ generates $H^2(\mbbcp_5;\mbbz)$ 
so that $x_2^k$ generates 
$H^{2k}(\mbbcp_5;\mbbz)$ additively. Once we compare these two ring structures, we 
decide that they are different, and hence these two spaces $G_2^+\mbbr^7$ and $\mbbcp_5$ are not homotopy equivalent. But that might not be very easy to show rigorously. 
That is why we choose another way to see it.

As we know that if two manifolds are homotopy equivalent, then they have the same homotopy groups. Exploiting this fact we can prove the following. 
			
\begin{thm}
$G_2^+\mathbb{R}^7$ is not homotopy equivalent to $\mathbb{CP}_5$. \label{G2R7nheCP5}
\end{thm}
\begin{proof}
We will be working on the classical Hopf fibration,
$$\mbb S^1 \to \mbb S^{11} \longrightarrow \mbbcp_5.$$
\ni The homotopy exact sequence of this fibration till 11-th level reads, 
$$\hspace{-3mm}0 \rightarrow \pi_{11}\,\mbb S^{11} \rightarrow \pi_{11} \,\mbbcp_5 \rightarrow 0\cdots
 0 \rightarrow \pi _ 3\,\mbbcp_5 \rightarrow 0 \rightarrow 0 \rightarrow \pi_2\,
\mbbcp_5 \rightarrow \pi_1\, \mbb S^1 \rightarrow 0 $$	
\ni	vanishing of the terms of which is provided by the vanishing of lower homotopy groups of the 11-sphere and higher homotopy groups of the circle. As a consequence 
we obtain the homotopy groups of the projective space as, 
$$\pi_{01234567891011}(\mbbcp_5) 
  =\left(0,0,\mathbb{Z},0,0,0,0,0,0,0,0,\mbbz\right)$$
\ni and the rest of the homotopy groups are determined by the 11-sphere's
$$\pi_k\,\mbbcp_5 = \pi_k \,\mbb S^{11}~~\tn{for}~~k\geq3.$$

\ni Next, recording 
the homotopy groups of the Stiefel manifold from \cite{hatchervb}, 
\begin{equation*}\pi _{012345}(V_2\mathbb{R}^7)=\left( 0,0,0,0,0,
		\mathbb{Z}_{2}\right) \end{equation*}
\ni the homotopy exact sequence of the Grassmannian fibration (\ref{G2R7fibration}) becomes the following
\bct $0 \to \mbbz_2 \to \pi_5\, G_2^+\mbbr^7 \to 
 0 \to 0 \to \pi_4\, G_2^+\mbbr^7 \to 
 0 \to 0 \to \pi_3\, G_2^+\mbbr^7 \to 
 0 \to 0 \to \pi_2\, G_2^+\mbbr^7 \to 
\pi_1\,\mbb S^1 \rightarrow 0 \rightarrow \pi _1\, G_2^+\mbbr^7 \to 0.$\ect
\ni This helps us to compute the homotopy groups up to fifth level as follows,  
\begin{equation}\pi_{012345}(G_2^+\mbbr^7)=(0,0,\mbbz,0,0,\mbbz_2).\label{homotopyG2R7}\end{equation}
First four fundamental groups of $\mathbb{CP}_5$ and $G_2^+\mathbb{R}^7$ are same but the fifth one is different. So these two spaces are not homotopy equivalent. 
		\end{proof}	


\section{Odd dimesional case}\label{secodd}					

In this section we will deal with the oriented real Grassmann manifold $G_2^+\mbbr^{2k+1}$ for $k\geq 2$. Similar calculations as in the previous section applies in the general odd dimensional case as well. We start with the Stiefel manifolds which have similar cohomology, 
$$H^*(V_2\mbbr^{2k+1};\mbbz)=\mbbz[x_{2k},x_{4k-1}]/\langle x_{2k}^2,x_{4k-1}^2,2x_{2k}, x_{2k}x_{4k-1} \rangle.$$
The cohomological Serre spectral sequence applied to the fibration 
\begin{equation*} \mbb S^1 \rightarrow V_2\mbbr^{2k+1} \longrightarrow G_2^+\mbbr^{2k+1} \label{G2R2k+1fibration} \end{equation*}
\ni has vanishing limit except for the entries $E_\infty^{0,0}=E_\infty^{4k,1}=\mbbz$ and $E_\infty^{2k,0}=\mbbz_2$. This forces the differentials $d_2^{i,1}$ 
in the second page to be isomorphism for all $i=-1\cdots 4k-1$ except for $2k-2$ which implies that our Grassmannian is a cohomological projective space $\mbbcp_{2k-1}$.  
Considering the limit of its image, differential on $E_\infty^{2k-2,1}$ has to be  multiplication by 2. Other relations similarly applies for this odd dimensional general case as well as the previous section. 
As a consequence we arrive at the following generalization. 

\begin{thm}\label{thmoddcuptruncated}	The cohomology ring of the Grassmann manifold is the following truncated polynomial ring for which $deg x_m=m$. 
\[ H^*( G_2^+\mbbr^{2k+1};\mbbz) \approx  
\mbbz[x_2,x_{2k}]/ \langle x_2^k-2x_{2k}, x_{2k}^2\rangle. \]\end{thm}
	
\ni This ring is different from that of the complex projective space which is 
$$H^*(\mbbcp_{2k-1};\mbbz)=\mbbz[x_2]/ \langle x_2^{2k}\rangle.$$
\ni On the other hand one can proceed in the homotopy groups direction as well. 
The Stiefel manifold $V_2\mbbr^{2k+1}$ is 2k-2 connected and the 2k-1-th homotopy group is $\mbbz_2$. Using this on the homotopy exact sequence of the circle bundle gives the non-zero homotopy groups $\mbbz$ and $\mbbz_2$ at the levels 2 and 2k-1 of the Grassmannian. 
The rest of the homotopy groups are again trivial upto these levels. 
Similarly using Hopf fibration homotopy exact sequence, the 
proof of the Theorem \ref{G2R7nheCP5} also suggest the following corollary.  

\beg{cor} The homotopy groups of the complex projective spaces are 
computed from the spheres as follows. 
$$\pi_j\,\mbbcp_n =\left\{\begin{array}{ll}
 0    & j < 2 \\ 
\mbbz & j = 2   \\ 
\pi_j\, \mbb S^{2n+1}  & j > 2
\end{array}   \right.$$\end{cor}

\ni So that first two nontrivial homotopy groups of $\mbbcp_{2k-1}$ lie at the levels 2 and 4k-1 which distinguish the homotopy groups of two manifolds at the level 2k-1.

\section{Even dimensional case}\label{seceven}					

The even dimensional case varies widely from the odd one in our situation. 
Major reason for diversity stems 
from the topology of Stiefel manifolds, 
cohomology of which can be extracted from \cite{hatcherat} as follows. 
$$H^*(V_2\mbbr^{2k};\mbbz)=
\mbbz[x_{2k-2},x_{2k-1}]/\langle x_{2k-2}^2,x_{2k-1}^2,x_{2k-2}x_{2k-1}\rangle.$$ 

\ni To be able to compute the cohomology of the Grassmannian we set up the cohomological Serre spectral sequence for the fibration, 
\begin{equation*} \mbb S^1 \rightarrow V_2\mbbr^{2k} \longrightarrow G_2^+\mbbr^{2k}. \label{G2R2kfibration} \end{equation*}
\ni The zeros coming from the Stiefel manifold in the limit with the top and 
bottom cohomology provides zeros at the odd levels and $\mbbz$'s at the even levels 
except for the mid-level. See Table \ref{table:E2CohomG28} for a low dimensional case. 
\beg{table}[ht]
\caption{ {\em The second page of the cohomological Serre spectral sequence for $G_2^+\mbb R^8$.}}
\bct{
$\begin{array}{cc|c|c|c|c|c|c|c|c|c|c|c|c|c|} \cline{3-15}
    &1&\mbbz a&0&\mbbz ax_2&0&\mbbz ax_4&0&\mbbz ax_6\oplus\mbbz ay_6&0&\mbbz ax_8&0&\mbbz ax_{10}&0&\mbbz ax_{12} \\ \cline{3-15}
E_2~~~ &0&\mbbz 1&0&\mbbz x_2&0&\mbbz  x_4&0&\mbbz x_6\oplus\mbbz y_6&0&\mbbz x_8&0&\mbbz x_{10}&0&\mbbz x_{12}\\ \cline{3-15}
& \multicolumn{1}{c}{ } & \multicolumn{1}{c}{0} 
& \multicolumn{1}{c}{1} & \multicolumn{1}{c}{2} 
& \multicolumn{1}{c}{3} & \multicolumn{1}{c}{4}
& \multicolumn{1}{c}{5} & \multicolumn{1}{c}{6} 
& \multicolumn{1}{c}{7} & \multicolumn{1}{c}{8}
& \multicolumn{1}{c}{9} & \multicolumn{1}{c}{10}
& \multicolumn{1}{c}{11} & \multicolumn{1}{@{}c@{}}{12} \\  
& \multicolumn{1}{c}{ } & \multicolumn{1}{c}{ } 
& \multicolumn{1}{c}{ } & \multicolumn{1}{c}{ } 
& \multicolumn{1}{c}{ } & \multicolumn{1}{c}{ }
& \multicolumn{1}{c}{ } & \multicolumn{1}{c}{2k-2} 
& \multicolumn{1}{c}{ } & \multicolumn{1}{c}{ }
& \multicolumn{1}{c}{ } & \multicolumn{1}{c}{ }
& \multicolumn{1}{c}{ } & \multicolumn{1}{@{}c@{}}{ }
\end{array}$   
\setlength{\unitlength}{1mm}
\begin{picture}(0,0)
\put(-12,17.5){\vector(2,-1){8}}
\end{picture} }
\ect\label{table:E2CohomG28}\end{table}

\ni We figure out that $E_\infty^{2k-3,1}=E_\infty^{2k-1,0}=0$ from our computation of the second page.  
Since $\mbbz=H^{2k-2}(V;\mbbz)=F^{0,2k-2}\supset F^{1,2k-3}\supset\cdots$ imposing 
$E_\infty^{0,2k-2}=\cdots=E_\infty^{2k-3,1}=0$ yields that 
$\mbbz=F^{1,2k-3}=\cdots=F^{2k-2,0}$ last term of which is equal to $E_\infty^{2k-2,0}$.

Similarly we have $\mbbz=H^{2k-1}(V;\mbbz)=F^{0,2k-1}\supset F^{1,2k-2}\supset\cdots$, imposing 
$E_\infty^{0,2k-1}=\cdots=E_\infty^{2k-3,2}=0$ yields that 
$\mbbz=F^{1,2k-3}=\cdots=F^{2k-2,1}$. But we know that  
$F^{2k-1,0}=0$ from the limit, so that  $E_\infty^{2k-2,1}=F^{2k-2,1}/F^{2k-1,0}=\mbbz$. The same argument is valid to prove $E_\infty^{4k-4,1}=\mbbz$ 
since we know $H^{4k-4}(V;\mbbz)=\mbbz$ and as an outside of territory limit $E_\infty^{4k-3,0}=0$. 
Combining these results we obtain Table \ref{table:EinftyCohomG28} as the 
limit of the spectral sequence. 

\beg{table}[ht]
\caption{{\em The limiting page of cohomological Serre spectral sequence for $G_2^+\mbb R^8$.}}
\bct{
$\begin{array}{cc|c|c|c|c|@{}c@{}|@{}c@{}|c|c|c|c|c|c|c|c} \cline{3-15}
           &1&0&0&0&0&0
           &\tcb 0&\tcb \mbbz&0&0&0&0&0&\tcb\mbbz \\ \cline{3-15}
E_\infty~~~&0&\mbbz&0&0&0&   0   &0
&\tcb{\mbbz}&\tcb{0}&0&0&0&0&0 & \tcb 0 \\ \cline{3-15}
& \multicolumn{1}{c}{ } & \multicolumn{1}{c}{0} 
& \multicolumn{1}{c}{1} & \multicolumn{1}{c}{2} 
& \multicolumn{1}{c}{3} & \multicolumn{1}{c}{4}
& \multicolumn{1}{c}{5} & \multicolumn{1}{c}{6} 
& \multicolumn{1}{c}{7} & \multicolumn{1}{c}{8}
& \multicolumn{1}{c}{9} & \multicolumn{1}{c}{10}
& \multicolumn{1}{c}{11} & \multicolumn{1}{@{}c@{}}{12} \\  
\end{array}$   }\ect \label{table:EinftyCohomG28}  \end{table}

\ni Let $H=H^{2k-2}(G;\mbbz)$ be the mid-level cohomology. Then since the differential  $d_2^{2k-4,1}$ is injective, we have an isomorphic copy of $\mbbz$ inside $H$. 
Since the limit $E_\infty^{2k-2,0}=\mbbz$ we have an isomorphism $H/d(\mbbz) \approx \mbbz$. This isomorphism will reveal the result. First of all $H$ contains no torsion. If there were a torsion subgroup then the intersection 
$T\cap d(\mbbz)$ has to be trivial, otherwise $d(\mbbz)$ should contain nontrivial elements of finite order which is impossible. So the torsion $T$ remains in the quotient. This means that it is a subgroup $T<\mbbz$ of the 
quotient which is free so $T$ has to be zero. So $H$ is free. It remains 
to determine its rank. If $rk H=1$ then the quotient has to be finite which is not the case. If the $rk H>2$ then one can show that the quotient has rank greater than 1. So the rank must be 2.

\ni To be able to figure out the cohomology at the mid-level alternatively, we need to proceed with auxiliary tools. 
Poincar\'e polynomial of the Grassmannian in the even case is given by \cite{ghv,gluckmackenziemorgan}, 
$$p_{G_2^+\mbbr^{2k}}(t)=1+t^2+t^4+\cdots +t^{2k-4}+ 2t^{2k-2}+t^{2k}+\cdots + t^{4k-4}.$$
\ni This shows that the free part has rank two. An application of the universal 
coefficient theorem links the torsion to a lower level so that
$$T^{2k-2}=\tn{Ext}(H_{2k-3}(G;\mbbz),\mbbz)=T_{2k-3}=0.$$ 
\ni This finally completes the Table \ref{table:E2CohomG28}. 
   
Analysing the product structure with similar techniques as in the previous section does not yield the cohomology ring in the even case.
We only get the relations $x_2x_l=x_{l+2}$ for $l\neq 2k-4,2k-2$. 
To figure out the rest of the relations we need to understand the intrinsic 
structure of the ring. In what follows $E$ and $F$ will denote $E_2^{2k}$ and $F_2^{2k}$, the canonical bundles over the Grassmann manifold $G_2^+\mbbr^{2k}$.  
They are obtained by taking the plane corresponding to a point and its orthogonal complement(or quotient) to produce vector bundles of rank 2 and 2k-2 over our Grassmannian. 
We collect the related results of \cite{shizhou} here as follows.


\beg{thm}\label{shizhouthm}
We have the following relations in the cohomology of the 

Grassmann manifold $G_2^+\mbbr^{2k}$.\\

(a) $_{\mbbcp_m}\int i^* e^m =\, _{\overline{\mbbcp}_m}\int i^* e^m = (-1)^m$.\\

(b) $_{\mbbcp_{k-1}}\int i^*eF =\, -\, _{\overline{\mbbcp}_{k-1}}\int i^*eF=1$.\\

(c) $[\mbbcp_m] \in H_{2m}(G;\mbbz)$ and $e^m \in H^{2m}(G;\mbbz)$ are the generators 

for $m<k-1$.\\

(d) $[\mbb G(2,m+2)] \in H_{2m}(G;\mbbz)$ and 
$e^m / 2 \in H^{2m}(G;\mbbz)$ are the generators 

for $m>k-1$.\\

(e) $[\mbbcp_{k-1}],[\ol{\mbbcp}_{k-1}] \in H_{2k-2}(G;\mbbz)$ and 
$( (-e)^{k-1}\pm eF ) /2 \in H^{2k-2}(G;\mbbz)$

are the generators. 
They also correspond to each other by Poincar\' e duality. \\

\end{thm}

\ni Using these characteristic classes and integrals we will be able to figure out the generators and relations in our Grassmannian. 
Now we are ready to compute the cohomology ring. 

\begin{thm}\label{thmevencuptruncated}	The cohomology ring of the Grassmann manifold is the following truncated polynomial ring for which $deg x_m=m$. 
\[\hspace{-1cm}
H^*( G_2^+\mbbr^{2k};\mbbz) \approx  
{\mbbz[x_2,x_{2k-2},y_{2k-2},x_{2k}] \tn{\Large $/$} \hspace{-1.5mm} \left\langle 
\begin{array}{c} 
x_{2k-2}+y_{2k-2}+(-1)^k x_2^{k-1}~,~ 
2x_{2k-2}^2-x_2^{2k-2} \\ [1.5\jot] 
2x_2x_{2k-2}+(-1)^k x_2^k~,~2y_{2k-2}^2 - x_2^{2k-1}\\ [1.5\jot]
x_{2k-2}y_{2k-2}~,~x_{2k}^2 ~,~ x_2^{2k-1}
\end{array}  \right\rangle  }. \]
\end{thm}

\beg{proof}
The cohomology of the manifold $G_2^+\mbbr^{2k}$ is additively generated by the elements, 
$$1~,~e~,~e^2 \, \cdots ~ e^{k-2} ~,~( (-e)^{k-1}\pm eF ) /2~,~ e^k/2 
~\cdots ~e^{2k-2}/2.$$
\ni Setting up the variables,
\beg{eqnarray}
x_2      & := & e \\ [2\jot]
2x_{2k-2} & := & (-e)^{k-1}+eF \label{variablex2k-2}\\ [2\jot]  
2y_{2k-2} & := & (-e)^{k-1}-eF \label{variabley2k-2}\\ [2\jot]
2x_{2k}   & := & e^{k-1}
\end{eqnarray}

\ni  Adding up the variables (\ref{variablex2k-2}) and (\ref{variabley2k-2}) we get the relation,
\begin{align}x_{2k-2}+y_{2k-2}=(-1)^{k-1}x_2^{k-1}.\label{relation1}\end{align}
\ni This provides the only relation among all these three variables. 
{\large
$${\renewcommand{\arraystretch}{2} 
\begin{array}{rcl} 
_{G_2^+\mbbr^{2k}}\int
x_{2k-2}^2 
& = & {1\over 4}\, _{G_2^+\mbbr^{2k}}\int \{ e^{2k-2}+e^2F+2(-e)^{k-1} \}  \\ 
& = & {1\over 4}\, 
 \{ 2 +\, _{\mbbcp_{k-1}-\overline{\mbbcp}_{k-1}}\int eF + 2(-e)^{k-1} \}  \\
& = & {1\over 4}\,  \{ 2 + (1--1) + 0 \}=1. \\
\end{array} }$$  
}
\ni by Theorem \ref{shizhouthm}. Since the rank of the cohomology in this dimension is one, this implies that 
\begin{align}2x_{2k-2}^2=x_2^{2k-2}.\label{relation2}\end{align}
\ni Multiplying the variables, we compute 
{\large
$${\renewcommand{\arraystretch}{2} 
\begin{array}{rcl} 
_{\mbb G(2,k+2)} 
\int
x_2 x_{2k-2} 
& = & {1\over 2}\, _{\mbb G(2,k+2)}
\int \{ e ( (-e)^{k-1}+eF ) \}  \\ 
& = & {1\over 2}\, _{G_2^+\mbbr^{2k}}
\int \{ e^{k-2} e ( (-e)^{k-1}+eF ) \}  \\
& = & {1\over 2}\, _{G_2^+\mbbr^{2k}}
\int \{ (-1)^{k-1} e^{2k-2} + e^{k-1}eF  \}  \\
& = & (-1)^{k-1} + {1\over 2}\, _{G_2^+\mbbr^{2k}}
\int \{ e^{k-1}eF  \}  \\
& = & (-1)^{k-1} + {1\over 2}\, _{\mbbcp_{k-1}-\overline{\mbbcp}_{k-1}}\int e^{k-1} \\
& = & (-1)^{k-1}
\end{array} }$$  
}
\ni which is the coefficient of the generator at the 2k-th level, so that we get the relation, 
\begin{align}x_2x_{2k-2}=(-1)^{k-1}e^k/2=2^{-1} (-1)^{k-1} x_2^k.\label{relation3}\end{align}
\ni Square of the remaining variable can be decomposed by virtue of the 
relation (\ref{relation1}) and also combining with the relations (\ref{relation2},\ref{relation3}) we have computed, 


{\large
\beg{eqnarray}
y_{2k-2}^2 
& = & ((-1)^{k-1}x_2^{k-1}-x_{2k-2})^2 \nonumber\\ [3\jot] 
& = & x_2^{2k-2}+x_{2k-2}^2-2(-1)^{k-1}x_2^{k-1}x_{2k-2} \nonumber\\ [3\jot]
& = & x_2^{2k-2}+x_2^{2k-2}/2 - 2(-1)^{k-1}x_2^{k-2} 2^{-1}(-1)^{k-1}x_2^k\nonumber\\ [3\jot]  
& = & x_2^{2k-1}/2 \label{relation4}
\end{eqnarray} }
\ni Finally we will compute the product in the middle level. 
{\large \beg{eqnarray}
_{G_2^+\mbbr^{2k}}\int x_{2k-2}y_{2k-2}  
& = & {1\over 4}\, _{G_2^+\mbbr^{2k}}\int \{ e^{2k-2}-e^2F \} \nonumber\\ [3\jot] 
& = & {1\over 2} - {1\over 4}\, _{\mbbcp_{k-1}-\overline{\mbbcp}_{k-1}}\int eF \nonumber\\ [3\jot]
& = & 0. \label{relation5} 
\end{eqnarray} }
\ni The rest of the relations are produced by either vanishing above the dimension or commutativity of forms.  
\end{proof}

\ni Note that in the lowest dimension when k=2, it is a nice exercise for the reader  
to establish an isomorphism with the cohomology of product of 2-spheres,  
$$H^*(\mbb S^2\times \mbb S^2)=\mbbz[x_2,y_2]/\langle x_2^2,y_2^2,x_2y_2-y_2x_2\rangle.$$

Considering the homotopy groups, the Stiefel manifold $V_2\mbbr^{2k}$ is 2k-3 connected and the 2k-2-nd homotopy 
group of it is $\mbbz$. The homotopy exact sequence of the circle bundle 
gives the non-zero homotopy group $\mbbz$ at the levels 2 and 2k-2 of the Grassmannian.  The rest of the groups are trivial upto these levels.  
On the other hand the first two nontrivial homotopy groups of $\mbbcp_{2k-2}$ 
are at the levels 2 and 4k-3 which distinguish the homotopy groups at the level 2k-2.




\bigskip


\bibliography{cp5}{}
\bibliographystyle{alphaurl}

\bigskip

{\small
\begin{flushleft}
\textsc{Mathematisches Institut, 
Rheinische Friedrich-Wilhelms-Universität Bonn
Endenicher Allee 60, 
D-53115 Bonn, 
Germany
.}\\ \vspace{2mm}
\textit{E-mail address:} 
\texttt{\textbf{kalafat@\,math.uni-bonn.de
}}
\end{flushleft}
}

\end{document}